\documentclass[a4paper,12pt,final]{amsart}
\usepackage{times,a4wide,mathrsfs,amssymb,dsfont}

\newcommand{\C}{\mathbb{C}}

\newcommand{\QQ}{\mathbb{Q}}
\newcommand{\NN}{\mathbb{N}}
\newcommand{\PP}{\mathbb{P}}

\newcommand{\OO}{\mathcal O}
\newcommand{\Ss}{\mathcal S}

\newcommand{\XX}{\mathcal X}

\newcommand{\LLL}{\mathcal L}
\newcommand{\NNN}{\mathcal N}

\newcommand{\MM}{\mathcal M}
\newcommand{\FF}{\mathcal F}

\newcommand{\wt}{\widetilde}
\newcommand{\ima}{\hbox{Im}}
\newcommand{\rom}{\romannumeral}

\newtheorem{theorem}{Theorem}[section]
\newtheorem{claim}[theorem]{Claim}
\newtheorem{lemma}[theorem]{Lemma}
\newtheorem{corollary}[theorem]{Corollary}
\newtheorem{proposition}[theorem]{Proposition}
\newtheorem{conjecture}[theorem]{Conjecture}
\newtheorem{nonumbering}{Theorem}

\newtheorem{nonumberingc}{Corollary}

\newtheorem{convention}{Conventions}

\theoremstyle{definition}
\newtheorem{remark}[theorem]{Remark}
\newtheorem{definition}[theorem]{Definition}

\newtheorem{nonumberingt}{Acknowledgements}

\begin{document}
\author[Robert Laterveer]
{Robert Laterveer}

\address{Institut de Recherche Math\'ematique Avanc\'ee,
CNRS -- Universit\'e 
de Strasbourg,\
7 Rue Ren\'e Des\-car\-tes, 67084 Strasbourg CEDEX,
FRANCE.}
\email{robert.laterveer@math.unistra.fr}

\title{On the Chow ring of certain Fano fourfolds}

\begin{abstract} We prove that certain Fano fourfolds of K3 type constructed by 
Fatighenti--Mongardi have a multiplicative Chow--K\"unneth decomposition. We present some consequences for the Chow ring of these fourfolds.
\end{abstract}



\maketitle

\section{Introduction}

This note is part of a program aimed at understanding the class of varieties admitting a {\em multiplicative Chow--K\"unneth decomposition\/}, in the sense of \cite{SV}.
The concept of multiplicative Chow--K\"unneth decomposition was introduced in order to better understand the (conjectural) behaviour of the Chow ring of hyperk\"ahler varieties, while also providing a systematic explanation of the peculiar behaviour of the Chow ring of K3 surfaces and abelian varieties.
In \cite{S2}, the following conjecture is raised:

\begin{conjecture}\label{conj} Let $X$ be a smooth projective Fano variety of K3 type (i.e. $\dim X=2m$ and the Hodge numbers $h^{p,q}(X)$ are $0$ for all $p\not=q$ except for $h^{m-1,m+1}(X)=h^{m+1,m-1}(X)=1$). Then $X$ has a multiplicative Chow--K\"unneth decomposition.
\end{conjecture}

This conjecture is verified in some special cases \cite{Ver}, \cite{d3}, \cite{S2}.
The aim of the present note is to provide some more evidence for Conjecture \ref{conj}. We consider two families of Fano fourfolds of K3 type (these are the families labelled B1 and B2 in \cite{FM}).

 \begin{nonumbering}[=Theorem \ref{main}] Let $X$ be a smooth fourfold of one of the following types:
  \begin{itemize}
  \item a hypersurface of multidegree $(2,1,1)$ in $M=\PP^3\times\PP^1\times\PP^1$;
  
  \item a hypersurface of multidegree $(2,1)$ in $M=\operatorname{Gr}(2,4)\times\PP^1$ (with respect to the Pl\"ucker embedding).
  \end{itemize}
  Then $X$ has a multiplicative Chow--K\"unneth decomposition.
 \end{nonumbering}
 
 Theorem \ref{main} has interesting consequences for the Chow ring $A^\ast(X)_{\QQ}$ of these fourfolds:
  
 \begin{nonumberingc}[=Corollary \ref{cor}] Let $X$ and $M$ be as in Theorem \ref{main}. 
 Let $R^3(X)\subset A^3(X)_{\QQ}$ be the subgroup generated by the Chern class $c_3(T_X)$, the image of the restriction map $A^3(M)_{\QQ}\to A^3(X)_{\QQ}$, and intersections 
 $A^1(X)_{\QQ}\cdot A^2(X)_{\QQ}$ of divisors with $2$-cycles.
 The cycle class map induces an injection
   \[ R^3(X)\ \hookrightarrow\ H^6(X,\QQ)\ .\]
  \end{nonumberingc}
 
  This is reminiscent of the famous result of Beauville--Voisin describing the Chow ring of a $K3$ surface \cite{BV}.
  More generally, there is a similar injectivity result for the Chow ring of certain self-products $X^m$ (Corollary \ref{cor}).
  
  Another consequence is the existence of a multiplicative decomposition in the derived category for families of Fano fourfolds as in Theorem \ref{main} (Corollary \ref{deldec}).
  
%
%
  \vskip0.6cm

\begin{convention} In this note, the word {\sl variety\/} will refer to a reduced irreducible scheme of finite type over $\C$. For a smooth variety $X$, we will denote by $A^j(X)$ the Chow group of codimension $j$ cycles on $X$ 
with $\QQ$-coefficients.

The notation 
$A^j_{hom}(X)$ will be used to indicate the subgroups of 
homologically trivial cycles.

For a morphism between smooth varieties $f\colon X\to Y$, we will write $\Gamma_f\in A^\ast(X\times Y)$ for the graph of $f$.
The contravariant category of Chow motives (i.e., pure motives with respect to rational equivalence as in \cite{Sc}, \cite{MNP}) will be denoted $\MM_{\rm rat}$. 

We will write $H^\ast(X):=H^\ast(X,\QQ)$ for singular cohomology with $\QQ$-coefficients.
\end{convention}
  
 \vskip0.6cm

  \section{The Fano fourfolds}
  
  \begin{proposition}\label{4folds}
  
  (\rom1) Let $X\subset\PP^3\times\PP^1\times\PP^1$ be a smooth hypersurface of multidegree $(2,1,1)$ (following \cite{FM}, we will say $X$ is ``of type B1''). Then $X$ is Fano, and the Hodge numbers of $X$ are
      \[ \begin{array}[c]{ccccccccccccc}
      &&&&&& 1 &&&&&&\\
      &&&&&0&&0&&&&&\\
      &&&&0&&3&&0&&&&\\
        &&&0&&0&&0&&0&&&\\
      &&0&&1&&22&&1&&0&&\\
      &&&0&&0&&0&&0&&&\\
       &&&&0&&3&&0&&&&\\
        &&&&&0&&0&&&&&\\      
        &&&&&& 1 &&&&&&\\
        \end{array}\]
        
     (\rom2) Let $X\subset\operatorname{Gr}(2,4)\times\PP^1$ be a smooth hypersurface of multidegree $(2,1)$ with respect to the Pl\"ucker embedding (following \cite{FM}, we will say $X$ is ``of type B2''). Then $X$ is Fano, and the Hodge numbers of $X$ are
      \[ \begin{array}[c]{ccccccccccccc}
      &&&&&& 1 &&&&&&\\
      &&&&&0&&0&&&&&\\
      &&&&0&&2&&0&&&&\\
        &&&0&&0&&0&&0&&&\\
      &&0&&1&&22&&1&&0&&\\
      &&&0&&0&&0&&0&&&\\
       &&&&0&&2&&0&&&&\\
        &&&&&0&&0&&&&&\\      
        &&&&&& 1 &&&&&&\\
        \end{array}\]
  \end{proposition}
  
  \begin{proof} An easy way to determine the Hodge numbers is to use the following identification:
  
  \begin{lemma}\label{blow} Let $Z$ be a smooth projective variety of Picard number 1, and $X\subset Z\times\PP^1$ a general hypersurface of bidegree $(d,1)$. Then
     $X$ is isomorphic to the blow-up of $Z$ with center $S$, where $S\subset Z$ is a smooth dimensionally transversal intersection of $2$ divisors of degree $d$. 
     
 Conversely, given a smooth dimensionally transversal intersection $S\subset Z$ of $2$ divisors of degree $d$, the blow-up of $Z$ with center $S$ is isomorphic to a smooth hypersurface $X\subset Z\times\PP^1$ of bidegree $(d,1)$.
     \end{lemma}
     
     \begin{proof} This is \cite[Lemma 2.2]{FM}. The gist of the argument is that $X$ determines a pencil of divisors in $Z$, of which $S$ is the base locus.
     In terms of equations, if $X$ is defined by $y_0 f + y_1 g=0$ (where $[y_0:y_1]\in\PP^1$ and $f, g\in H^0(Z,\OO_Z(d))$) then $S$ is defined by $f=g=0$.
     It follows that for $X$ general (in the usual sense of ``being parametrized by a Zariski open in the parameter space'') the locus $S$ is smooth.
     \end{proof}
     
     In case (\rom1), $Z=\PP^3\times\PP^1$ and $S$ is a genus 7 K3 surface. In case (\rom2), $Z=\operatorname{Gr}(2,4)$ (which is a quadric in $\PP^5$), and $S$ is a genus 5 K3 surface. This readily gives the Hodge numbers. 
  \end{proof}

  \begin{remark} Fatighenti--Mongardi \cite{FM} give a long list of Fano varieties of K3 type. 
  The Fano fourfolds of Proposition \ref{4folds}(\rom1) and (\rom2) are labelled B1 resp. B2 in their list.
    \end{remark}
  
  \section{Multiplicative Chow--K\"unneth decomposition}

	\begin{definition}[Murre \cite{Mur}]\label{ck} Let $X$ be a smooth projective
		variety of dimension $n$. We say that $X$ has a 
		{\em CK  decomposition\/} if there exists a decomposition of the
		diagonal
		\[ \Delta_X= \pi^0_X+ \pi^1_X+\cdots +\pi^{2n}_X\ \ \ \hbox{in}\
		A^n(X\times X)\ ,\]
		such that the $\pi^i_X$ are mutually orthogonal idempotents and the action of
		$\pi^i_X$ on $H^j(X)$ is the identity for $i=j$ and zero for $i\not=j$.
		Given a CK decomposition for $X$, we set 
		$$A^i(X)_{(j)} := (\pi_X^{2i-j})_\ast A^i(X).$$
		The CK decomposition is said to be {\em self-dual\/} if
		\[ \pi^i_X = {}^t \pi^{2n-i}_X\ \ \ \hbox{in}\ A^n(X\times X)\ \ \ \forall
		i\ .\]
		(Here ${}^t \pi$ denotes the transpose of a cycle $\pi$.)
		
		  (NB: ``CK decomposition'' is short-hand for ``Chow--K\"unneth
		decomposition''.)
	\end{definition}
	
	\begin{remark} \label{R:Murre} The existence of a Chow--K\"unneth decomposition
		for any smooth projective variety is part of Murre's conjectures \cite{Mur},
		\cite{MNP}. 
		It is expected that for any $X$ with a CK
		decomposition, one has
		\begin{equation*}\label{hope} A^i(X)_{(j)}\stackrel{??}{=}0\ \ \ \hbox{for}\
		j<0\ ,\ \ \ A^i(X)_{(0)}\cap A^i_{num}(X)\stackrel{??}{=}0.
		\end{equation*}
		These are Murre's conjectures B and D, respectively.
	\end{remark}

	\begin{definition}[Definition 8.1 in \cite{SV}]\label{mck} Let $X$ be a
		smooth
		projective variety of dimension $n$. Let $\Delta_X^{sm}\in A^{2n}(X\times
		X\times X)$ be the class of the small diagonal
		\[ \Delta_X^{sm}:=\bigl\{ (x,x,x) : x\in X\bigr\}\ \subset\ X\times
		X\times X\ .\]
		A CK decomposition $\{\pi^i_X\}$ of $X$ is {\em multiplicative\/}
		if it satisfies
		\[ \pi^k_X\circ \Delta_X^{sm}\circ (\pi^i_X\otimes \pi^j_X)=0\ \ \ \hbox{in}\
		A^{2n}(X\times X\times X)\ \ \ \hbox{for\ all\ }i+j\not=k\ .\]
		In that case,
		\[ A^i(X)_{(j)}:= (\pi_X^{2i-j})_\ast A^i(X)\]
		defines a bigraded ring structure on the Chow ring\,; that is, the
		intersection product has the property that 
		\[  \ima \Bigl(A^i(X)_{(j)}\otimes A^{i^\prime}(X)_{(j^\prime)}
		\xrightarrow{\cdot} A^{i+i^\prime}(X)\Bigr)\ \subseteq\ 
		A^{i+i^\prime}(X)_{(j+j^\prime)}\ .\]
		
		(For brevity, we will write {\em MCK decomposition\/} for ``multiplicative Chow--K\"unneth decomposition''.)
	\end{definition}
	
	\begin{remark}
	The property of having an MCK decomposition is
	severely restrictive, and is closely related to Beauville's ``(weak) splitting
	property'' \cite{Beau3}. For more ample discussion, and examples of varieties
	admitting a MCK decomposition, we refer to
	\cite[Chapter 8]{SV}, as well as \cite{V6}, \cite{SV2},
	\cite{FTV}, \cite{LV}.
	\end{remark}

	\begin{remark}\label{self} It turns out that any MCK decomposition is self-dual, cf. \cite[Footnote 24]{FV}.
	\end{remark}

There are the following useful general results:	

\begin{proposition}[Shen--Vial \cite{SV}]\label{product} Let $M,N$ be smooth projective varieties that have an MCK decomposition. Then the product $M\times N$ has an MCK decomposition.
\end{proposition}

\begin{proof} This is \cite[Theorem 8.6]{SV}, which shows more precisely that the {\em product CK decomposition\/}
  \[ \pi^i_{M\times N}:= \sum_{k+\ell=i} \pi^k_M\times \pi^\ell_N\ \ \ \in A^{\dim M+\dim N}\bigl((M\times N)\times (M\times N)\bigr) \]
  is multiplicative.
\end{proof}

\begin{proposition}[Shen--Vial \cite{SV2}]\label{blowup} Let $M$ be a smooth projective variety, and let $f\colon\wt{M}\to M$ be the blow-up with center a smooth closed subvariety
$N\subset M$. Assume that
\begin{enumerate}

\item $M$ and $N$ have an MCK decomposition;

\item the Chern classes of the normal bundle $\NNN_{N/M}$ are in $A^\ast_{(0)}(N)$;

\item the graph of the inclusion morphism $N\to M$ is in $A^\ast_{(0)}(N\times M)$;

\item the Chern classes $c_j(T_M)$ are in $A^\ast_{(0)}(M)$.

\end{enumerate}
Then $\wt{M}$ has an MCK decomposition, the Chern classes $c_j(T_{\wt{M}})$ are in $A^\ast_{(0)}(\wt{M})$, and the graph $\Gamma_f$ is in $A^\ast_{(0)}(  \wt{M}\times M)$.
\end{proposition}

\begin{proof} This is \cite[Proposition 2.4]{SV2}. (NB: in loc. cit., $M$ and $N$ are required to have a {\em self-dual\/} MCK decomposition; however, the self-duality is actually a redundant hypothesis, cf. Remark \ref{self}.)

In a nutshell, the construction of loc. cit. is as follows. Given MCK decompositions $\pi^\ast_M$ and $\pi^\ast_N$ (of $M$ resp. $N$), one defines
  \begin{equation}\label{pimck}
    \pi^j_{\wt{M}}:= \Psi\circ\Bigl( \pi^j_M\oplus \bigoplus_{k=1}^r \pi^{j-2k}_N\Bigr)\circ \Psi^{-1}\  \ \in\ A^{\dim \wt{M}}(\wt{M}\times \wt{M})\ ,\end{equation}
    where $r+1$ is the codimension of $N$ in $M$, and $\Psi, \Psi^{-1}$ are certain explicit correspondences (this is \cite[Equation (13)]{SV2}). Then one checks that the $\pi^\ast_{\wt{M}}$ form an MCK decomposition.
\end{proof}

 \section{Main result}
 
 \begin{theorem}\label{main} Let $X$ be a smooth fourfold of one of the following types:
  \begin{itemize}
  \item a hypersurface of multidegree $(2,1,1)$ in $\PP^3\times\PP^1\times\PP^1$;
  
  \item a hypersurface of multidegree $(2,1)$ in $\operatorname{Gr}(2,4)\times\PP^1$ (with respect to the Pl\"ucker embedding).
  \end{itemize}
 
 Then $X$ has an MCK decomposition. Moreover, the Chern classes $c_j(T_X)$ are in 
 $A^\ast_{(0)}(X)$.
 \end{theorem} 
 
 \begin{proof} The argument relies on the alternative description of the general $X$ given by Lemma \ref{blow}.
 
 
 
 \medskip\noindent
 {\it Step 1:} We restrict to $X$ sufficiently general, in the sense that $X$ is a blow-up as in Lemma \ref{blow} with {\em smooth\/} center $S$.
  
 To construct an MCK decomposition for $X$, we apply the general Proposition \ref{blowup}, with $M$ being either $\PP^3\times\PP^1$ or $\operatorname{Gr}(2,4)$, and  $N$ being the K3 surface $S\subset M$ determined by Lemma \ref{blow}. All we need to do is to check that the assumptions of Proposition \ref{blowup} are met with.
 
 Assumption (1) is verified, since both varieties with trivial Chow groups $M$ and K3 surfaces $S$ have an MCK decomposition. For $M$ there is no choice involved ($M$ has a unique CK decomposition which is MCK). For $S$, we choose 
   \begin{equation}\label{k3}  \pi^0_S:=\mathfrak{o}_S\times S\ ,\ \pi^4_S:=S\times \mathfrak{o}_S\ ,\ \pi^2_S:=\Delta_S-\pi^0_S-\pi^4_S\ \ \ \in\ A^2(S\times S)\ ,\end{equation}
   where $\mathfrak{o}_S\in A^2(S)$ is the distinguished zero-cycle of \cite{BV}.
 This is an MCK decomposition for $S$ \cite[Example 8.17]{SV}. 
  
 Assumption (4) is trivially satisfied: one has $A^\ast_{hom}(M)=0$ and so (because $\pi^j_M$ acts as zero on $H^{2i}(M)$ for $j\not=2i$) one has $A^\ast_{}(M)=A^\ast_{(0)}(M)$.
 
 To check assumptions (2) and (3), we consider things family-wise. That is, we write
   \[ \bar{B}:= \PP H^0\bigl( M, \LLL^{\oplus 2}\bigr) \ ,\]
   where the line bundle $\LLL$ is either $\OO_M(2,1)$ (in case $M=\PP^3\times\PP^1$) or $\OO_M(2)$ (in case $M=\operatorname{Gr}(2,4)$),
   and we consider the universal complete intersection
   \[\bar{\Ss}\ \to\ \bar{B}\ .\] 
   We write $B_0\subset\bar{B}$ for the Zariski open parametrizing smooth dimensionally transversal intersections, and $\Ss\to B_0$ for the base change (so the fibres $S_b$ of $\Ss\to B_0$ are exactly the K3 surfaces that are the centers of the blow-up occurring in Lemma \ref{blow}).
   We now make the following claim:
   
   \begin{claim}\label{gfc} Let $\Gamma\in A^i(\Ss)$ be such that 
     \[ \Gamma\vert_{S_b}=0\ \ \ \hbox{in}\ H^{2i}(S_b)\ \ \ \forall b\in B_0\ .\]
     Then also
      \[ \Gamma\vert_{S_b}=0\ \ \ \hbox{in}\ A^i(S_b)\ \ \ \forall b\in B_0\ .\]
     \end{claim}
     
     We argue that the claim implies that assumptions (2) and (3) of Proposition \ref{blowup} are met with (and thus Proposition \ref{blowup} can be applied to prove Theorem \ref{main}).
     Indeed, let $p_j\colon \Ss\times_{B_0} \Ss\to \Ss$, $j=1,2$, denote the two projections.
     We observe that
     \[ \pi^0_\Ss:={1\over 24} (p_1)^\ast  c_2(T_{\Ss/{B_0}})\ ,\ \ \pi^4_\Ss:={1\over 24} (p_2)^\ast  c_2(T_{\Ss/{B_0}})\ ,\ \ \pi^2_\Ss:= \Delta_\Ss-\pi^0_\Ss-\pi^4_\Ss\ \ \ \in A^4(\Ss\times_{B_0} \Ss) \]
     defines a ``relative MCK decomposition'', in the sense that for any $b\in B_0$, the restriction $\pi^i_\Ss\vert_{S_b\times S_b}$ defines an MCK decomposition for $S_b$ which agrees with (\ref{k3}).
     
     Let us now check that assumption (2) is satisfied. Since $A^1(S_b)=A^1_{(0)}(S_b)$, we only need to consider $c_2$ of the normal bundle. That is,
     we need to check that for any $b\in B_0$ there is vanishing
     \begin{equation}\label{need} (\pi^2_{S_b})_\ast c_2(\NNN_{S_b/M})\stackrel{??}{=}0\ \ \ \hbox{in}\ A^2(S_b)\ .\end{equation}
     But we can write
      \[  (\pi^2_{S_b})_\ast c_2(\NNN_{S_b/M})  =  \Bigl(  (\pi^2_\Ss)_\ast c_2(\NNN_{\Ss/(M\times B_0)})\Bigr)\vert_{S_b}    \ \ \ \hbox{in}\ A^2(S_b)\ \]
      (for the formalism of relative correspondences, cf. \cite[Chapter 8]{MNP}),
      and besides we know that $  (\pi^2_{S_b})_\ast c_2(\NNN_{S_b/M})$ is homologically trivial ($\pi^2_{S_b}$ acts as zero on $H^4(S_b)$). Thus, Claim 
      \ref{gfc} implies the necessary vanishing (\ref{need}). 
      
     Assumption (3) is checked similarly. Let $\iota_b\colon S_b\to M$ and $\iota\colon \Ss\to M\times B$ denote the inclusion morphisms.
     To check assumption (3), we need to convince ourselves of the vanishing    
      \begin{equation}\label{need2} (\pi^\ell_{S_b\times M})_\ast (\Gamma_{\iota_b})\stackrel{??}{=}0\ \ \ \hbox{in}\ A^4(S_b\times M)\ \ \ \forall \ell\not= 8\ ,\ \forall b\in B_0\ .\end{equation}
      
     Since $\Gamma_{\iota_b}\in A^4(S_b\times M)$, one knows that $ (\pi^\ell_{S_b\times M})_\ast (\Gamma_{\iota_b})$ is homologically trivial for any
     $\ell\not=8$. Furthermore, we can write the cycle we are interested in as the restriction of a universal cycle:
      \[  (\pi^\ell_{S_b\times M})_\ast (\Gamma_{\iota_b}) =   \Bigl( (\sum_{j+k=\ell} \pi^j_\Ss\times \pi^k_{M}) _\ast (\Gamma_\iota)\Bigr)\vert_{S_b\times M}
      \ \ \ \hbox{in}\ A^4(S_b\times M)\ .\]
      For any $b\in B_0$, there is a commutative diagram
      \[  \begin{array}[c]{ccc}
              A^4(\Ss\times M) &\to&    A^4(S_b\times M)      \\
              &&\\
            \ \ \   \downarrow{\cong}&&\ \ \  \downarrow{\cong}\\
            &&\\
             \bigoplus A^\ast(\Ss) &\to&  \bigoplus A^\ast(S_b)\\
             \end{array} \]
             where horizontal arrows are restriction to a fibre, and
             where vertical arrows are isomorphisms because $M$ has trivial Chow groups.
            Claim \ref{gfc} applied to the lower horizontal arrow shows the vanishing (\ref{need2}), and so assumption (3) holds.
            
            It is only left to prove the claim. Since $A^i_{hom}(S_b)=0$ for $i\le 1$, the only non-trivial case is $i=2$. 
            Given $\Gamma\in A^2(\Ss)$ as in the claim, let $\bar{\Gamma}\in A^2(\bar{\Ss})$ be a cycle restricting to $\Gamma$.
            We consider the two projections
             \[ \begin{array}[c]{ccc}
       \bar{\Ss}&\xrightarrow{\pi}& M  \\
       \ \ \ \ \downarrow{\scriptstyle \phi}&&\\
         \ \  \bar{B}\ &&\\
          \end{array}\]  
      Since any point of $M$ imposes exactly one condition on $\bar{B}$, the morphism $\pi$ has the structure of a projective bundle. As such, any 
      $\bar{\Gamma}\in A^{2}(\bar{\Ss})$ can be written
          \[ \bar{\Gamma}=    \sum_{\ell=0}^2 \pi^\ast( a_\ell)  \cdot \xi^\ell  \ \ \ \hbox{in}\ A^{2}(\bar{\Ss})\ ,\]
                where $a_\ell\in A^{2-\ell}( M)$ and $\xi\in A^1(\bar{\Ss})$ is the relative hyperplane class.
                  
        Let $h:=c_1(\OO_{\bar{B}}(1))\in A^1(\bar{B})$. There is a relation
        \[  \phi^\ast(h)=\alpha \xi +  \pi^\ast(h_1)\ \ \ \hbox{in}\ A^1(\bar{\Ss})\ ,\]
        where $\alpha\in\QQ$ and $h_1\in A^1(M)$. As in \cite[Proof of Lemma 1.1]{PSY}, one checks that $\alpha\not=0$ (if $\alpha$ were $0$, we would have
        $\phi^\ast(h^{\dim \bar{B}})=\pi^\ast(h_1^{\dim \bar{B}})$, which is absurd since $\dim \bar{B}>4$ and so the right-hand side must be $0$).
        Hence, there is a relation
        \[ \xi = {1\over \alpha} \bigl(\phi^\ast(h)-\pi^\ast(h_1)\bigr)\ \ \ \hbox{in}\ A^1(\bar{\Ss})\ .\]        
      For any $b\in B_0$, the restriction of $\phi^\ast(h)$ to the fibre $S_b$ vanishes, and so it follows that
      \[ \bar{\Gamma}\vert_{S_b} = a_0^\prime\vert_{S_b}\ \ \ \hbox{in}\ A^{2}(S_b)\ \]
     for some $a_0^\prime\in A^2( M)$. But
       $A^2( M)$ is generated by intersections of divisors in case $M=\PP^3\times\PP^1$, and $A^2(M)$ is generated by divisors and $c_2$ of the tautological bundle in case $M=\operatorname{Gr}(2,4)$. In both cases, it follows that
                   \[ \bar{\Gamma}\vert_{S_b} = a_0^\prime\vert_{S_b} \ \ \in\ \QQ [{\mathfrak{o}}_{S_b}]\ \ \ \subset\ A^{2}(S_b)\ ,\]
  (in the second case, this is proven as in \cite[Proposition 2.1]{PSY}). 
      Given ${\Gamma}\in A^2({\Ss})$ a cycle such that the fibrewise restriction has degree zero, this shows that the fibrewise restriction is zero in $A^2(S_b)$. 
 Claim \ref{gfc} is proven.
 
 \medskip
 \noindent
 {\it Step 2:} It remains to extend to {\em all\/} smooth hypersurfaces as in the theorem. That is, let $B\subset\bar{B}$ be the open such that the Fano fourfold $X_b$ (which is the blow-up of $M$ with center $S_b$) is smooth. One has $B\supset B_0$. Let $\XX\to B$ and $\XX^0\to B_0$ denote the universal families of Fano fourfolds over $B$ resp. $B_0$. 
 
 From step 1, one knows that $X_b$ has an MCK decomposition for any $b\in B_0$. A closer look at the proof reveals more: the family $\XX^0\to B_0$ has a
 ``universal MCK decomposition'', in the sense that there exist relative correspondences $\pi^\ast_{\XX^0}\in A^4(\XX^0\times_{B_0}\XX^0)$ such that for each $b\in B_0$ the restriction $\pi^\ast_{X_b}:=\pi^\ast_{\XX^0}\vert_b\in A^4(X_b\times X_b)$ forms an MCK decomposition for $X_b$. (To see this, one observes that Proposition \ref{blow} is ``universal'': given families $\MM\to B$, $\NNN\to B$ and universal MCK decompositions $\pi^\ast_\MM$, $\pi^\ast_\NNN$, the result of (\ref{pimck}) is a universal MCK decomposition for $\wt{\MM}\to B$.)
 
 A standard argument now allows to spread out the MCK property from $B_0$ to the larger base $B$. That is, we define
   \[ \pi^j_\XX:= \bar{\pi}^j_{\XX^0}\ \ \in \ A^4(\XX\times_B \XX)\ \]
  (where $\bar{\pi}$ refers to the closure of a representative of $\pi$). The ``spread lemma'' \cite[Lemma 3.2]{Vo} (applied to $\XX\times_B \XX$) gives that the $\pi^\ast_\XX$ are a fibrewise CK decomposition, and the same spread lemma (applied to $\XX\times_B \XX\times_B \XX$) gives that the $\pi^\ast_\XX$ are a fibrewise MCK decomposition. This ends step 2. 
 \end{proof}
 
 \begin{remark}\label{franch} Claim \ref{gfc} states that the families $\Ss\to B_0$ verify the ``Franchetta property'' as studied in \cite{FLV}. It is worth mentioning that the Franchetta property for the universal K3 surface of genus $g\le 10$ (and for some other values of $g$) was already proven in \cite{PSY}; the families considered in Claim \ref{gfc} are different, however, so Claim \ref{gfc} is not covered by \cite{PSY} (e.g., in case $M=\PP^3\times\PP^1$ the K3 surfaces of Claim \ref{gfc} have Picard number at least $2$, so they correspond to a Noether--Lefschetz divisor in $\FF_7$).
 
 As a corollary of Claim \ref{gfc}, the universal families $\XX\to B$ of Fano fourfolds of type B1 or B2 also verify the Franchetta property. (Indeed, in view of \cite[Lemma 3.2]{Vo} it suffices to prove this for $\XX^0\to B_0$.
 In view of Lemma \ref{blow}, $\XX^0$ can be
 constructed as the blow-up of $M\times B_0$ with center $\Ss$. This blow-up yields a relative correspondence from $\XX^0$ to $\Ss$, inducing a commutative diagram
   \[ \begin{array}[c]{ccc}
        A^j(\XX^0) &\to& A^{j-1}(\Ss)\oplus \bigoplus \QQ\\
        &&\\
        \downarrow&&\downarrow\\
        &&\\
        A^j(X_b) &\to& A^{j-1}(S_b)\oplus \bigoplus \QQ\\ 
        \end{array}\]
        where horizontal arrows are injective (by the blow-up formula).
        The Franchetta property for $\Ss\to B_0$ thus implies the Franchetta property for $\XX^0\to B_0$.)
 \end{remark}

 \begin{remark} One would expect that for Fano varieties of K3 type, there is a {\em unique\/} MCK decomposition. I cannot prove this unicity for the Fano fourfolds of Theorem \ref{main}. (At least, one may observe that for $X$ as in Theorem \ref{main} the induced splitting of the Chow ring is canonical, since $A^j(X)=A^j_{(0)}(X)$ for $j\not=3$, whereas $A^3(X)=A^3_{(0)}(X)\oplus A^3_{(2)}(X)$ with $A^3_{(0)}=A^2(X)\cdot A^1(X)$ and $A^3_{(2)}(X)=A^3_{hom}(X)$). 
 \end{remark}

 \section{Some consequences} 
  
  \subsection{An injectivity result}
  
 \begin{corollary}\label{cor} Let $X$ and $M$ be as in Theorem \ref{main}, and let $m\in\NN$. Let $R^\ast(X^m)\subset A^\ast(X^m)$ be the $\QQ$--subalgebra
   \[ \begin{split} R^\ast(X^m):=  < (p_i)^\ast A^1(X), (p_i)^\ast A^2(X), (p_{ij})^\ast(\Delta_X), (p_i)^\ast c_j(T_X),& \\
                                    (p_i)^\ast \ima\bigl( A^i(M)\to A^i(X)&\bigr)>\ \ \ \subset\ A^\ast(X^m)\ .\\
                                    \end{split}\]
   (Here $p_i\colon X^m\to X$ and $p_{ij}\colon X^m\to X^2$ denote projection to the $i$th factor, resp. to the $i$th and $j$th factor.)
   
  The cycle class map induces injections
   $ R^j(X^m) \hookrightarrow H^{2j}(X^m)$
   in the following cases:
   
   \begin{enumerate}
   
   \item $m=1$ and $j$ arbitrary;
   
   \item $m=2$ and $j\ge 5$;
   
   \item $m=3$ and $j\ge 9$.
   \end{enumerate}
       \end{corollary}
 
 \begin{proof} Theorem \ref{main}, in combination with Proposition \ref{product}, ensures that $X^m$ has an MCK decomposition, and so $A^\ast(X^m)$ has the structure of a bigraded ring under the intersection product. The corollary is now implied by the conjunction of the two following claims:

\begin{claim}\label{c1} There is inclusion
  \[ R^\ast(X^m)\ \ \subset\ A^\ast_{(0)}(X^m)\ .\]
  \end{claim}
  
 \begin{claim}\label{c2} The cycle class map induces injections
   \[ A^j_{(0)}(X^m)\ \hookrightarrow\ H^{2j}(X^m)\ \] 
   provided $m=1$, or $m=2$ and $j\ge 5$, or $m=3$ and $j\ge 9$.
\end{claim}

To prove Claim \ref{c1}, we note that $A^k_{hom}(X)=0$ for $k\not= 3$, which readily implies the equality $A^k(X)=A^k_{(0)}(X)$ for $k\not= 3$. The fact that $c_3(T_X)$ is in $A^3_{(0)}(X)$ is part of Theorem \ref{main}. The fact that $\Delta_X\in A^4_{(0)}(X\times X)$ is a general fact for any $X$ with a (necessarily self-dual) MCK decomposition \cite[Lemma 1.4]{SV2}. It remains to prove that codimension three
cycles coming from the ambient space $M$ are in $A^3_{(0)}(X)$. To this end, we observe that such cycles are universally defined, i.e.
  \[ \ima\Bigl( A^3(M)\to A^3(X)\Bigr)\ \ \subset\ \ima\Bigl( A^3(\XX)\ \to\ A^3(X)\Bigr)\ ,\]
  where $\XX\to B$ is the universal family as before, and $X=X_{b_0}$ for some $b_0\in B$. Given $a\in A^3(\XX)$, applying the Franchetta property (Remark \ref{franch}) to 
   \[ \Gamma:= (\pi^j_\XX)_\ast(a)\ \ \ \in A^3(\XX)\ ,\ \ \ j\not=6\ ,\]
   one finds that the restriction $a\vert_X\in A^3(X)$ lives in $A^3_{(0)}(X)$. In particular, it follows that
   \[  \ima\Bigl( A^3(M)\to A^3(X)\Bigr)\ \ \subset\ A^3_{(0)}(X)\ ,\]   
   as desired.
Since the projections $p_i$ and $p_{ij}$ are pure of grade $0$ \cite[Corollary 1.6]{SV2}, and $A^\ast_{(0)}(X^m)$ is a ring under the intersection product, this proves Claim \ref{c1}.

To prove Claim \ref{c2}, we observe that Manin's blow-up formula \cite[Theorem 2.8]{Sc} gives an isomorphism of motives
  \[ h(X)\cong h(S)(1)\oplus \bigoplus {\mathds{1}}(\ast)\ \ \ \hbox{in}\ \MM_{\rm rat}\ .\]
  Moreover, in view of Proposition \ref{blowup} (cf. also \cite[Proposition 2.4]{SV2}), the correspondence inducing this isomorphism is of pure grade $0$.
  
 In particular, for any $m\in\NN$ we have isomorphisms of Chow groups
  \[ A^j(X^m)\cong A^{j-m}(S^m)\oplus \bigoplus_{k=0}^4  A^{j-m+1-k}(S^{m-1})^{b_k} \oplus \bigoplus A^\ast(S^{m-2})\oplus \bigoplus_{\ell\ge 3} A^\ast(S^{m-\ell})  \ ,  \] 
  and this isomorphism respects the $A^\ast_{(0)}()$ parts. Claim \ref{c2} now follows from the fact that for any surface $S$ with an MCK decomposition, and any $m\in\NN$, the cycle class map induces injections
  \[ A^i_{(0)}(S^m)\ \hookrightarrow\ H^{2i}(S^m)\ \ \ \forall i\ge 2m-1\ \]
  (this is noted in \cite[Introduction]{V6}, cf. also \cite[Proof of Lemma 2.20]{acs}).

\end{proof}

\subsection{Decomposition in the derived category}
 Given a smooth projective morphism $\pi\colon \XX\to B$, Deligne \cite{Del} has proven a decomposition in the derived category of sheaves of $\QQ$-vector spaces on $B$:
    \begin{equation}\label{del}  R\pi_\ast\QQ\cong \bigoplus_i  R^i \pi_\ast\QQ[-i]\ .\end{equation}
    As explained in \cite{Vdec}, for both sides of this isomorphism there is a cup-product: on the right-hand side, this is the direct sum of the usual cup-products of local systems, while on the left-hand side, this is the derived cup-product (inducing the usual cup-product in cohomology). In general, the isomorphism (\ref{del}) is {\em not\/} compatible with these cup-products, even after shrinking the base $B$ (cf. \cite{Vdec}). In some rare cases, however, there is such a compatibility (after shrinking): this is the case for families of abelian varieties \cite{DM}, and for families of K3 surfaces \cite{Vdec}, \cite[Section 5.3]{Vo} (cf. also \cite[Theorem 4.3]{V6} and \cite[Corollary 8.4]{FTV} for some further cases).
    
Given the close link to K3 surfaces, it is not surprising that the Fano fourfolds of Theorem \ref{main} also have such a multiplicative decomposition: 
   
 \begin{corollary}\label{deldec} Let $\XX\to B$ be a family of Fano fourfolds of type B1 or B2. There is a non-empty Zariski open $B^\prime\subset B$, such that the isomorphism (\ref{del}) becomes multiplicative after shrinking to $B^\prime$.
 \end{corollary}
 
 \begin{proof} This is a formal consequence of the existence of a relative MCK decomposition, cf. \cite[Proof of Theorem 4.2]{V6} and \cite[Section 8]{FTV}.
 \end{proof}
   
Given a family $\XX\to B$ and $m\in\NN$, let us write $\XX^{m/B}$ for the $m$-fold fibre product 
  \[ \XX^{m/B}:=\XX\times_B\XX\times_B\cdots\times_B \XX\ .\]
  Corollary \ref{deldec} has the following concrete consequence, which is similar to a result for families of K3 surfaces obtained by Voisin \cite[Proposition 0.9]{Vdec}:
  
  \begin{corollary}\label{deldec2} Let $\XX\to B$ be a family of Fano fourfolds of type B1 or B2.   
 Let $z\in A^r(\XX^{m/B})$ be a polynomial in (pullbacks of) divisors and codimension $2$ cycles on $\XX$.  
  Assume the fibrewise restriction $z\vert_b$ is homologically trivial, for some $b\in B$. Then there exists a non-empty Zariski open $B^\prime\subset B$
  such that
    \[ z=0\ \ \ \hbox{in}\ H^{2r}\bigl((\XX^\prime)^{m/B^\prime},\QQ\bigr)\ .\]
    \end{corollary}
    
    \begin{proof} The argument is the same as \cite[Proposition 0.9]{Vdec}. First, one observes that divisors $d_i$ and codimension $2$ cycles $e_j$ on $\XX$ admit a cohomological decomposition (with respect to the Leray spectral sequence)
        \[ \begin{split}  d_i&= d_{i0} + \pi^\ast(d_{i2})\ \ \ \hbox{in}\ H^0(B, R^2\pi_\ast\QQ)\oplus \pi^\ast H^2(B,\QQ) \cong H^2(\XX,\QQ)\ ,\\
                                e_j&= e_{j0} +\pi^\ast(e_{j2}) +\pi^\ast(e_{j4})\ \ \ \hbox{in}\  H^0(B, R^4\pi_\ast\QQ)\oplus \pi^\ast H^2(B)^{\oplus 2}  \oplus \pi^\ast H^4(B) \cong H^4(\XX,\QQ)\ .\\  
                        \end{split}\]
      We claim that the cohomology classes $d_{ik}$ and $e_{jk}$ are {\em algebraic\/}. This claim implies the corollary: indeed, given a polynomial
      $z=p(d_i,e_j)$, one may take $B^\prime$ to be the complement of the support of 
      the cycles $d_{i2}$, $e_{j2}$ and $e_{j4}$. Then over the restricted base one has equality
      \[  z:=p(d_i,e_j)= p(d_{i0},e_{j0})\ \ \ \hbox{in}\ H^{2r}\bigl((\XX^\prime)^{m/B^\prime},\QQ\bigr)\ .\]
      Multiplicativity of the decomposition ensures that (after shrinking the base some more)
      \[  p(d_{i0},e_{j0})\ \ \in\ H^0(B^\prime, R^{2r}(\pi^m)_\ast\QQ)\ \ \subset\ H^{2r}\bigl((\XX^\prime)^{m/B^\prime},\QQ\bigr)\ ,\]
      and so the conclusion follows.
      
      The claim is proven for divisor classes $d_i$ in \cite[Lemma 1.4]{Vdec}. For codimension $2$ classes $e_j$, the argument is similar to loc. cit.:
      let $h\in H^2(\XX)$ be an ample divisor class, and let $h_0$ be the part that lives in $H^0(B,R^2\pi_\ast\QQ)$. One has
      \[ e_j (h_0)^4 =  e_{j0}  (h_0)^4 +\pi^\ast(e_{j2})  (h_0)^4+\pi^\ast(e_{j4})  (h_0)^4\ \ \ \hbox{in}\ H^{12}(\XX,\QQ)\ .\]
      By multiplicativity, after some shrinking of the base the first two summands are contained in $H^0(B^\prime, R^{12}\pi_\ast\QQ)$, resp. in $H^2(B^\prime, R^{10}\pi_\ast\QQ)$, hence they are zero as $\pi$ has $4$-dimensional fibres. The above equality thus simplifies to
      \[ e_j (h_0)^4 =    \pi^\ast(e_{j4})  (h_0)^4\ \ \ \hbox{in}\ H^{12}(\XX,\QQ)\ .\]
      Pushing forward to $B^\prime$, one obtains
      \[ \pi_\ast (e_j (h_0)^4)= \pi_\ast \bigl((h_0)^4\bigr) e_{j4} = \lambda\, e_{j4}\ \ \ \hbox{in}\ H^4(B^\prime)\ ,\]
      for some $\lambda\in\QQ^\ast$. As the left-hand side is algebraic, so is $e_{j4}$.
      
      Next, one considers
       \[  e_j (h_0)^3 =  e_{j0}  (h_0)^3 +\pi^\ast(e_{j2})  (h_0)^3+\pi^\ast(e_{j4})  (h_0)^3\ \ \ \hbox{in}\ H^{10}(\XX,\QQ)\ .\]      
       The first summand is again zero for dimension reasons, and so
       \[ \pi^\ast(e_{j2})  (h_0)^3 =    e_j (h_0)^3 -  \pi^\ast(e_{j4})  (h_0)^3\ \ \in\ H^{10}(\XX,\QQ)\ \]
       is algebraic. A fortiori, $\pi^\ast(e_{j2})  (h_0)^4$ is algebraic, and so
       \[ \pi_\ast ( \pi^\ast (e_{j2})  (h_0)^4) =    \pi_\ast \bigl((h_0)^4\bigr) e_{j2} = \mu\, e_{j2}\ \ \ \hbox{in}\ H^2(B^\prime)\ ,  \ \ \ \mu\in\QQ^\ast\ ,\]          
       is algebraic.
      \end{proof}

\vskip1cm
\begin{nonumberingt} Thanks to two referees for constructive comments. Thanks to Len-boy from pandavriendjes.fr.

\end{nonumberingt}


\vskip1cm
\begin{thebibliography}{dlPG99}


\bibitem{Beau3} A. Beauville, On the splitting of the Bloch--Beilinson filtration, in: Algebraic cycles and motives (J. Nagel and C. Peters, editors), London Math. Soc. Lecture Notes 344, Cambridge University Press 2007,

\bibitem{BV} A. Beauville and C. Voisin, On the Chow ring of a K3 surface, J. Alg. Geom. 13 (2004), 417--426,



\bibitem{Del} P. Deligne, Th\'eor\`eme de Lefschetz et crit\`eres de d\'eg\'en\'erescence de suites spectrales, Inst. Hautes Etudes Sci.
Publ. Math. (1968), 259---278,

\bibitem{DM} C. Deninger and J. Murre, Motivic decomposition of abelian schemes and the Fourier transform, J. Reine Angew.
Math. 422 (1991), 201---219,



\bibitem{FM} E. Fatighenti and G. Mongardi, Fano varieties of K3 type and IHS manifolds, arXiv:1904.05679,


\bibitem{FLV} L. Fu, R. Laterveer and Ch. Vial, The generalized Franchetta conjecture for some hyper--K\"ahler varieties (with an appendix joint with M. Shen), to appear in Journal Math. Pures et  Appliqu\'ees,

\bibitem{FTV} L. Fu, Zh. Tian and Ch. Vial, Motivic hyperK\"ahler resolution conjecture: I. Generalized Kummer varieties, arXiv:1608.04968,

\bibitem{FV} L. Fu and Ch. Vial, Distinguished cycles on varieties with motive
		of abelian type and the Section Property, arXiv:1709.05644v2, to appear in J. Alg. Geom.,
		




\bibitem{acs} R. Laterveer, Algebraic cycles on some special hyperk\"ahler varieties, Rendiconti di Matematica e delle sue applicazioni 38 no. 2 (2017), 243---276, 

\bibitem{d3} R. Laterveer, A remark on the Chow ring of K\"uchle fourfolds of type $d3$, Bulletin Australian Math. Soc.,
https://doi.org/10.1017/S0004972719000273,

\bibitem{Ver} R. Laterveer, Algebraic cycles and Verra fourfolds, to appear in Tohoku Math. J.,

\bibitem{S2} R. Laterveer, On the Chow ring of Fano varieties of type $S2$, preprint,

\bibitem{LV} R. Laterveer and Ch. Vial, On the Chow ring of Cynk--Hulek Calabi--Yau varieties and Schreieder varieties, Canadian Journal of Math.,
		
\bibitem{Mur} J. Murre, On a conjectural filtration on the Chow groups of an algebraic variety, parts I and II, Indag. Math. 4 (1993), 177--201,

\bibitem{MNP} J. Murre, J. Nagel and C. Peters, Lectures on the theory of pure motives, Amer. Math. Soc. University Lecture Series 61, Providence 2013,


\bibitem{PSY} N. Pavic, J. Shen and Q. Yin, On O'Grady's generalized Franchetta conjecture, Int. Math. Res. Notices (2016), 1---13,

\bibitem{Sc} T. Scholl, Classical motives, in: Motives (U. Jannsen et alii, eds.), Proceedings of Symposia in Pure Mathematics Vol. 55 (1994), Part 1,

\bibitem{SV} M. Shen and Ch. Vial, The Fourier transform for certain hyperK\"ahler fourfolds, Memoirs of the AMS 240 (2016), 
		
\bibitem{SV2} M. Shen and Ch. Vial, On the motive of the Hilbert cube $X^{[3]}$, Forum Math. Sigma 4 (2016),
		
\bibitem{V6} Ch. Vial, On the motive of some hyperk\"ahler varieties, J. f\"ur Reine u. Angew. Math. 725 (2017), 235--247,

\bibitem{Vdec} C. Voisin, Chow rings and decomposition theorems for families of $K3$ surfaces and Calabi--Yau hypersurfaces,
Geometry and Topology 16 (2012), 433---473,

\bibitem{Vo} C. Voisin, Chow Rings, Decomposition of the Diagonal, and the Topology of Families, Princeton University Press, Princeton and Oxford, 2014.


\end{thebibliography}
\end{document}